\documentclass[11pt]{amsart}

\usepackage{amsthm,amssymb}

\newtheorem{theorem}{Theorem}[section]

\newtheorem{lemma}[theorem]{Lemma}

\newtheorem{prop}[theorem]{Proposition}
\newtheorem{coro}[theorem]{Corollary}

\theoremstyle{definition}

\newtheorem{rema}[theorem]{Remark}

\newcommand{\PSH}{{\rm PSH}}
\newcommand{\psh}{{\rm PSH}}
\newcommand{\vol}{{\rm Vol}}

\newcommand{\setdef}{\; | \;}

\title[Stability of Complex Monge-Amp\`ere equations]{Stability and H\"older regularity of solutions to complex Monge-Amp\`ere equations on compact  Hermitian manifolds}

\author{Chinh H. Lu}
\address{Universit\'e Paris-Saclay, CNRS, Laboratoire de Math\'ematiques d'Orsay, 91405, Orsay, France.}
\email{hoang-chinh.lu@universite-paris-saclay.fr}

\author{Trong-Thuc Phung}
\address{Ho Chi Minh City University of Technology, VNU-HCM, Vietnam}
\email{ptrongthuc@hcmut.edu.vn}

\author{T{\^a}t-Dat T\^o}
\address{\'Ecole Nationale de l'Aviation Civile\\
Unversit\'e  de Toulouse\\
 7, Avenue Edouard Belin\\
FR-31055 Toulouse Cedex 04, France}
\email{tat-dat.to@enac.fr}

\address{Current address of T\^at-Dat T\^o: Institut de Math\'ematiques de Jussieu-Paris Rive Gauche\\
Sorbonne Universit\'e - Campus Pierre et Marie Curie\\
4, place Jussieu \\
75252 Paris Cedex 05 France}
\email{tat-dat.to@imj-prg.fr}


 \usepackage{hyperref}
\hypersetup{
    bookmarks=true,         
    unicode=false,          
    pdftoolbar=true,        
    pdfmenubar=true,        
    pdffitwindow=false,     
    pdfstartview={FitH},   
    pdftitle={},    
    pdfauthor={},     
    colorlinks=true,      
   linkcolor=blue,        
   citecolor=blue,    
    filecolor=blue,    
    urlcolor=blue}

\begin{document}
\begin{abstract}
  Let $(X,\omega)$ be a compact Hermitian manifold. We establish a stability result for solutions to  complex Monge-Amp\`ere equations  with right-hand side in $L^p$, $p>1$.  Using this we prove that the solutions are H\"older continuous with the same exponent as in the K\"ahler case by Demailly-Dinew-Guedj-Ko{\l}odziej-Pham-Zeriahi. Our techniques also apply to the setting of big cohomology classes on compact K\"ahler manifolds.  
\end{abstract}

\maketitle

\tableofcontents

\section{Introduction}

One of the central problems in complex geometry is the existence of canonical metrics. On compact K\"ahler manifolds this question is intimately related to the study of complex Monge-Amp\`ere equations. Culminating with Yau's work \cite{Yau78}, which solves Calabi's conjecture, complex Monge-Amp\`ere equations have been studied and generalized in several directions with many important geometric applications. 

An essential step in solving complex Monge-Amp\`ere equations on compact manifolds is the uniform $L^{\infty}$ estimate. In Yau's work \cite{Yau78}, it was achieved via  Moser iteration process. Twenty years later,  Ko{\l}odziej \cite{Kol98} gave a novel proof using pluripotential theory which has been applied to many geometric situations. In the recent breakthrough of X.X. Chen and J. Cheng \cite{ChCh1,ChCh2,ChCh3}, pluripotential estimates of Ko{\l}odziej \cite{Kol98} and B{\l}ocki, see \cite{Bl05China,Bl11China}, were used to obtain a uniform estimate along the continuity path introduced earlier by X.X. Chen \cite{Chen_2015_AMQ}. 

In this paper we shall study complex Monge-Amp\`ere equations on compact (non-K\"ahler) Hermitian manifolds $(X,\omega)$ of dimension $n$, 
\begin{equation}
\label{eq: MA}
(\omega +dd^c u)^n = c f \omega^n,
\end{equation}
where $0\leq f\in L^p(X)$, for some $p>1$, and $c$ is a positive constant.  Here, if nothing is stated, the $L^p$-norm is computed with respect to the volume form $\omega^n$. Unlike  the K\"ahler case, here we have an extra variable: the constant $c$ which is in general not determined by $X,\omega$.  The non-degenerate case, i.e. when $0<f$ is smooth, has been studied by Cherrier \cite{Cher87}, Guan-Li \cite{GL10} under restrictive conditions. The general case was recently proved by Tosatti and Weinkove \cite{TW10}: there exists a unique constant $c=c_f >0$ and a unique (modulo adding a constant) smooth function $u$ with $\omega +dd^c u >0$, solving  \eqref{eq: MA}.  

In the last decade, pluripotential theory on compact Hermitian manifolds has been developed intensively  by S. Dinew, S. Ko{\l}odziej, and N-C. Nguy{\^e}n (see \cite{DK12ALM}, \cite{KN15Phong}, \cite{Ng16AIM}, \cite{Diw16AFST}).  The main difficulty in the Hermitian setting is that the comparison principle, which in the K\"ahler setting says that, for bounded $\omega$-psh functions $u,v$, 
\begin{equation}
	\label{eq: comparison principle}
	\int_{\{u<v\}} \omega_v^n \leq \int_{\{u<v\}} \omega_u^n,
\end{equation}
is missing. Nevertheless, a replacement for this, called the ``modified comparison principle'' was established in \cite{KN15Phong} which is a key tool in proving the existence of continuous solutions  \cite[Theorem 5.8]{KN15Phong}.  The  uniqueness of the constant $c$ was later proved in  \cite{Ng16AIM}. 

\medskip
Our first main result is a generalization of  \cite{GLZAIF} to the Hermitian setting. 
\begin{theorem}\label{thm: GLZ Hermitian}
Fix $0\leq  f,g \in L^p(X,\omega^n)$, $p>1$ such that $\int_X f \omega^n >0$ and $\int_X g \omega^n >0$. 
Assume that $u$ and $v$ are bounded $\omega$-psh functions on $X$ satisfying
$$
(\omega +dd^c u)^n = e^u f \omega^n \ \text{and} \ (\omega +dd^c v)^n = e^{v} g \omega^n. 
$$
Then for some constant $C>0$ depending on $X, \omega,n,p$, an upper bound  for $\|f\|_p, \|g\|_p$ and a positive lower bound for $\int_X f^{1/n} \omega^n$, $\int_X g^{1/n}\omega^n$, we have 
$$
\|u - v\|_{L^{\infty}(X)} \leq C \|f-g\|_p^{1/n}. 
$$ 
\end{theorem}

The proof of Theorem \ref{thm: GLZ Hermitian} goes along the same lines as in \cite{GLZAIF}.    An immediate consequence of  Theorem \ref{thm: GLZ Hermitian} is  a stability estimate for the constant $c$:

\begin{coro}\label{cor: stability of MA constant}
Assume that $0\leq f, g\in L^p(X)$ for some $p>1$. Then 
\[
|c_f -c_g|\leq C \|f-g\|_p^{1/n},
\]
where $C>0$ is a constant depending on $(X,\omega,n,p)$, an upper  bound for $\|f\|_p, \|g\|_p$ and a positive lower bound for $\|f\|_{1/n}$, $\|g\|_{1/n}$.
\end{coro}
Using Theorem \ref{thm: GLZ Hermitian} we can greatly improve the stability exponent in  \cite[Theorem A]{KN19} :

\begin{theorem} \label{thm: stability exponent}
Assume that $u,v$ are $\omega$-psh continuous solutions to 
$$
(\omega+dd^c u)^n = f\omega^n \ ;  \ (\omega+dd^c v)^n =g\omega^n, \ \sup_X u= \sup_X v=0,
$$
where $f,g\in L^p(X), p>1$ and $f\geq c_0>0$. Then 
$$
\|u-v\|_{\infty} \leq C \|f-g\|_p^{1/n},
$$
where $C$ depends on $X,\omega, n, p, c_0$ and an upper bound for $\|f\|_p, \|g\|_p$, and a positive lower bound for $\|g\|_{1/n}$. 
\end{theorem}
Compared to \cite{KN19}   the exponent is improved to be the same as in the K\"ahler case \cite{DZ10AIM}, but we still assume that $f\geq c_0>0$ for some positive constant $c_0$. It is interesting to know whether our techniques can be applied to treat more general right-hand sides considered in \cite{KN20}.   Improving the stability exponent is an interesting question because, at least, the stability estimate can be used to prove the H\"older continuity of solutions.  Moreover, in the recent breakthrough of Chen-Donaldson-Sun \cite{Chen_Donaldson_Sun_2015_JAMS} the H\"older continuity of solutions to degenerate complex Monge-Amp\`ere equations was exploited. 

If the comparison principle \eqref{eq: comparison principle} holds on  $(X,\omega)$ then many arguments from the K\"ahler case can be employed. In particular,  Theorem \ref{thm: stability exponent} holds without the strict positivity condition. Our argument also applies to the more general case of  big cohomology classes, improving a stability result of Guedj-Zeriahi \cite{GZ12MRL}: 

\begin{theorem}\label{thm: stability GZ}
	Let $(X,\omega)$ be a compact K\"ahler manifold of dimension $n$. Fix a closed smooth real $(1,1)$-form $\theta$ whose  cohomology class $\{\theta\}$ is big. Assume that $0\leq f,g \in L^p(X,\omega^n)$ are such that $\int_X f\omega^n = \int_X g \omega^n = {\rm Vol}(\theta)=1$.  If $u$ and $v$ are $\theta$-psh functions with minimal singularities on $X$ such that 
	$$
	\theta_u^n = f\omega^n, \ \theta_v^n = g\omega^n, \ \sup_X u= \sup_X v=0,
	$$
	then, for some constant $C>0$ depending on $(X,\omega,n, \theta, p)$ and an upper bound for $\|f\|_p, \|g\|_p$, we have 
	$$
	\sup_X |u-v|\leq C \|f-g\|_p^{1/n}.
	$$
\end{theorem}

Compared to \cite[Theorem C]{GZ12MRL}, here we have improved the exponent from $\frac{1}{2^n(n+1)-1}$ to $\frac{1}{n}$. Note that one can replace the $L^p$ norm by the $L^1$ norm and the exponent becomes slightly smaller (see \cite[Remark 2.2]{GLZAIF}).   The proof of Theorem \ref{thm: stability GZ} uses \cite[Theorem A]{GLZAIF} and Ko{\l}odziej's techniques as in \cite[Theorem 4.1]{Kolodziej_2003_IUMJ}. The main point here is that using \cite[Theorem A]{GLZAIF} we reduce the problem to the case in which the two densities $f,g$ are very close to each other in the following sense: $e^{-\varepsilon}f \leq g \leq e^{\varepsilon} f$, for some small constant $\varepsilon>0$.  In case $\theta$ is additionally semipositive  we get the same exponent as in  \cite{DZ10AIM}.  Our arguments also apply to the setting of prescribed singularities, where instead of asking for $u,v$ to have minimal singularities we ask $u,v$ to have the same singularity type as a given model potential \cite{DDL2,DDL4,DDL5}. 

\medskip 

A classical use of such  stability estimates is in proving H\"older continuity of solutions. Given $0\leq f\in L^p$ and $u$ a continuous solution to $\omega_u^n = f\omega^n$, it was proved by S. Ko{\l}odziej and N.C. Nguy{\^e}n \cite[Theorem B]{KN19} that if $f\geq c_0>0$ then $u$ is H\"older continuous. The strict positivity assumption was relaxed by the same authors recently in \cite{KN18AIF}, but the exponent is not optimal. Also,  due to the lack of uniqueness, the result in \cite{KN18AIF} does not give that all solutions are H\"older continuous. In the K\"ahler case, the H\"older continuity was first proved by Ko{\l}odziej \cite{Kol08MA} and improved by Demailly-Dinew-Guedj-Ko{\l}odziej-Pham-Zeriahi \cite{DDGKPZ14} using Demailly's approximation theorem \cite{Dem94}.  Related questions on H\"older continuity of solutions to complex Monge-Amp\`ere equations on compact K\"ahler manifolds have been studied by many authors. T.C. Dinh and V.A. Nguy{\^e}n \cite{DN14JFA} proved that a probability measure $\mu$ admits a H\"older continuous solution $\varphi_{\mu}$, i.e. $\varphi_{\mu}$ is $\omega$-psh and $(\omega +dd^c \varphi_{\mu})^n = \mu$, if and only if the super-potential associated to $\mu$ is H\"older continuous. The notion of super-potentials was introduced by T.C. Dinh and N. Sibony \cite{DS10JAG}. Using this notion, D.V. Vu has established in \cite{Vu18PAMS} a H\"older stability of families of Monge-Amp\`ere measures of H\"older continuous potentials.  He has also studied in \cite{Vu18MA}  H\"older continuity of potentials of probability measures supported in  real $\mathcal{C}^3$ submanifolds.   The study of H\"older continuity of solutions to complex Monge-Amp\`ere equations has many important applications in complex dynamics, we refer the reader to \cite{DNS10JDG} for more details. 

\medskip
  Using Theorem \ref{thm: GLZ Hermitian} we prove that any bounded solution to \eqref{eq: MA} is H\"older continuous with exponent in $(0, p_n)$. The constant $p_n$ here is the same as the one obtained in the K\"ahler case in \cite{DDGKPZ14}.

\begin{theorem}\label{thm: Holder exponent}
Let $(X,\omega)$ be a compact Hermitian manifold of dimension $n$. Fix  $0\leq f \in L^p(X), p>1$ with $\int_X f \omega^n >0$. Then any  solution $u$ to $\omega_u^n =c_f f \omega^n$ is H\"older continuous with H\"older exponent in $(0, p_n)$, where $p_n= \frac{2}{nq+1}$.
\end{theorem}
Here, $q$ is the conjugate of $p$, i.e. $1/p + 1/q =1$. The proof strictly follows \cite{KN19} and \cite{DDGKPZ14} in which the stability estimate  is used.  The only difference is that we use Theorem \ref{thm: GLZ Hermitian} to construct the perturbation functions, allowing to avoid the technical assumption $f\geq c_0>0$. Interestingly, our method also increases the H\"older exponent by a factor $n$ compared to \cite[Theorem B]{KN19}.  

In the last part of the paper we adapt the techniques of \cite{BD12} to establish H\"older regularity of plurisubharmonic envelopes, see Theorem \ref{thm: Holder envelope}.

 \medskip
 
\noindent {\bf Organization of the paper.}
 	In Section \ref{sect: backgrounds} we collect several known tools in pluripotential theory on compact Hermitian manifolds. The stability results will be proved in Section \ref{sect: stability}, while Theorem  \ref{thm: Holder exponent} will be proved in Section \ref{sect: Holder continuity}. 

\medskip

\noindent {\bf Acknowledgements. } We thank V\v{a}n-D\^ong Nguy{\^e}n for reading the first version of this paper and giving many useful comments. We are indebted to Ahmed Zeriahi for his very important help concerning Lemma \ref{lem: GKZ08}. We thank Vincent Guedj and the referee for many useful suggestions  which helped to improve the presentation of the paper.  We thank Ngoc-Cuong Nguy\^en for pointing out an error in the proof of Theorem \ref{thm: improvement of GZ12} in a previous version of the paper.  

C.H.Lu is supported by the CNRS project PEPS ``Jeune chercheuse, jeune chercheur''.   T.T. Phung is supported by Ho Chi Minh City University of Technology under grant number T-KHUD-2020-32. T.D. T\^o is  partially supported by the IEA project PLUTOCHE.

\section{Backgrounds}\label{sect: backgrounds}
Fix $(X,\omega)$ a compact Hermitian manifold of dimension $n$. 
In this section we review some background material in pluripotential theory on compact Hermitian manifolds. 
For a detailed treatment  we refer the reader to \cite{DK12ALM}, \cite[Section 1]{KN15Phong} and the recent surveys \cite{Diw16AFST}, \cite{Kol17IJM}.

A function $u : X \rightarrow \mathbb{R} \cup \{-\infty\}$ is quasi plurisubharmonic if locally it is the sum of a smooth and a psh function. We say that $u$ is $\omega$-psh if $u$ is quasi-psh and $\omega +dd^c u \geq 0$ in the sense of currents. Here, $d= \partial + \bar{\partial}$ and $d^c = i(\bar{\partial} -\partial)$ are real differential operators so that $dd^c = 2i \partial \bar{\partial}$.  We let $\PSH(X,\omega)$ denote the set of all $\omega$-psh functions on $X$ which are not identically $-\infty$. It follows from Demailly's approximation theorem  \cite{Dem92} that any $\omega$-psh function can be approximated from above by smooth strictly $\omega$-psh functions.

For a bounded $\omega$-psh function $u$, the complex Monge-Amp\`ere operator $\omega_u^n$ is defined by the method of Bedford and Taylor \cite{BT76}.  It was proved in \cite[Remark 5.7]{KN15Phong} that $\int_X \omega_u^n >0$, if $u$ is bounded. 

The main difficulty in the Hermitian setting is that the total mass of the Monge-Amp\`ere measure $\omega_u^n$ depends on the function $u$. This is why the comparison principle does not hold in general. It was proved in \cite{KN15Phong} that the following replacement for the comparison principle holds. 

\begin{theorem}[Modified comparison principle]\cite[Theorem 2.3]{KN15Phong}
	Let $u,v \in \PSH(X,\omega)\cap L^{\infty}(X)$.  Fix $0<\varepsilon <1$ and set 
	\[
	m_{\varepsilon}:= \inf_X (u-(1-\varepsilon)v).
	\]
	Then for all $0<s<\frac{\varepsilon^3}{16B}$, 
	$$
	\int_{\{u<(1-\varepsilon)v +m_{\varepsilon}+s\}} \omega_{(1-\varepsilon)v}^n \leq \left ( 1+ \frac{Cs}{\varepsilon^n}\right ) \int_{\{u<(1-\varepsilon)v+m_{\varepsilon}+s\}} \omega_{u}^n,
	$$
	where $C>0$ is a constant depending on $n,B$.
\end{theorem}
The constant $B$ depends only on $(X,\omega,n)$, it is chosen so that
$$
\begin{cases}
-B \omega^2 \leq 2n dd^c \omega \leq B \omega^2 \\
-B \omega^3 \leq 4n^2 d\omega \wedge d^c \omega \leq B \omega^3
\end{cases}.
$$
Note that the modified comparison principle is only valid on very small sublevel sets. This local analysis is suitable for proving the domination principle. The proof of this  result is (implicitly) written in  \cite[Lemma 2.3]{Ng16AIM}. In the K\"ahler case, the domination principle was proved by Dinew (see \cite[Proposition A1]{BL12PA})  using his uniqueness result \cite{DiwJFA09} (see \cite[Proposition 2.21]{Dar18S}, and \cite{LN19} for a different  proof using the envelope technique). 

\begin{prop}
	If $u,v$ are bounded $\omega$-psh functions such that $\omega_u^n(u<v) =0$ then $u\geq v$. 
\end{prop}
\begin{proof}
Assume by contradiction that $U:=\{u<v\}$ is not empty and set $m_{\varepsilon}:= \inf_X (u-(1-\varepsilon)v)$, for  $\varepsilon\in [0,1)$. Since $v$ is bounded and $m_0<0$, we see that for $\varepsilon>0$ small enough $m_{\varepsilon}<m_0/2<0$. Set $U(\varepsilon,s):= \{u<(1-\varepsilon)v + m_{\varepsilon} +s\}$. Then for $s>0$ and $\varepsilon>0$ small enough we have $U(\varepsilon,s) \subset U$. Hence by the modified comparison principle we have
$$
\varepsilon^n \int_{U(\varepsilon,s)} \omega^n \leq \int_{U(\varepsilon,s)} \omega_{(1-\varepsilon)v}^n \leq \left (1 +\frac{Cs}{\varepsilon^n}\right )\int_{U(\varepsilon,s)} \omega_u^n =0. 
$$
It follows that, for such choice of $s,\varepsilon$, $\int_{U(\varepsilon,s)} \omega^n=0$, hence $U(\varepsilon,s)= \emptyset$ which is a contradiction.
\end{proof}

Using the modified comparison principle, it was proved in   \cite[Lemma 2.3]{Ng16AIM} that subsolutions are smaller than supersolutions for $L^p$-density. The same proof applies to give the following:

\begin{prop}\cite{Ng16AIM}
\label{prop: sub and super solutions}
Assume that $u$ and $v$ are bounded $\omega$-psh functions such that 
$$
\omega_u^n \geq e^{\lambda(u-v)} \omega_v^n,
$$
for some constant $\lambda >0$. 
Then $u\leq v$. 
\end{prop}

Yet another application of the modified comparison principle yields the following minimum principle: 
\begin{prop}\cite[Proposition 2.5]{KN19}, \cite[Corollary 2.4]{Ng16AIM}
\label{prop: min principle}
Assume that $u$ and $v$ are continuous $\omega$-psh functions such that $\omega_u^n \leq c \omega_v^n$ on an open set $\Omega \subset X$. If $c<1$ then $\Omega\neq X$ and 
$$
\min_{\Omega} (u-v) = \min_{\partial \Omega} (u-v).
$$
\end{prop}

As shown in \cite{KN15Phong}, given $0\leq f \in L^p$ with $\int_X f \omega^n >0$, there exist a unique constant $c_f>0$ and $u\in \PSH(X,\omega)\cap L^{\infty}(X)$ such that $\omega_u^n = c_ff \omega^n$. The density $f$ is MA-admissible if $c_f=1$. The total mass of an admissible density in $L^p$ is uniformly controlled from below. 

\begin{prop}\cite[Proposition 2.4]{KN19}, \cite[Proposition 2.7]{KN20}
\label{prop: control mass from below}
Fix a constant $A_0>1$. Then there exists a constant $V_{\min}>0$ depending on $(X,\omega,n, A_0)$ such that for any MA-admissible $0\leq f\in L^p$ with $\|f\|_p\leq A_0$, we have 
$$
\int_X f \omega^n \geq 2^{n+1}V_{\min}. 
$$
\end{prop}
Although the total mass of $\omega_u^n$ depends on $u$, we can control the total mass of the Laplacian of $u$ by using a Gauduchon metric.

\begin{lemma}
\label{lem: Gauduchon}
There exists a uniform constant $C>0$ such that 
$$
C^{-1} \leq \int_X \omega_u \wedge \omega^{n-1} \leq C, \ \forall u \in \PSH(X,\omega) \cap L^{\infty}(X). 
$$
\end{lemma}

\begin{proof}
Let $G$ be a smooth function on $X$ such that $dd^c (e^G \omega^{n-1})=0$. The existence of $G$ follows from \cite{Gau77}. Using Stokes' theorem we then have 
$$
\int_X \omega_u \wedge (e^G \omega^{n-1}) = \int_X e^G \omega^n,
$$
from which the estimates follow. 
\end{proof}

\section{Stability of solutions}\label{sect: stability}

\subsection{On Hermitian manifolds}

We first extend the elliptic stability theorem in \cite{GLZAIF} to the non-K\"ahler case. 

\begin{theorem}\label{thm: GLZ Hermitian proof}
Fix $0\leq  f,g \in L^p(X,\omega^n)$, $p>1$ such that $\int_X f \omega^n >0$ and $\int_X g \omega^n >0$. 
Assume that $u,v$ are bounded $\omega$-psh functions on $X$ satisfying
$$
(\omega +dd^c u)^n = e^u f \omega^n \ \text{and} \ (\omega +dd^c v)^n = e^{v} g \omega^n. 
$$
Then for some constant $C>0$ depending on $X, \omega,n,p$, an upper bound  for $\|f\|_p, \|g\|_p$ and a positive lower bound for $\|f\|_{1/n}$, $\| g\|_{1/n}$, we have 
\begin{equation}\label{eq: stability}
    |u - v| \leq C \|f-g\|_p^{1/n}. 
\end{equation}
\end{theorem}

\begin{rema}
As shown in \cite[Remark 2.2]{GLZAIF} the $L^p$-norm can be replaced by the $L^1$-norm and the exponent becomes $1/(n+\varepsilon)$, where $\varepsilon>0$ is arbitrarily small.  
\end{rema}

\begin{proof}
 The proof  uses a perturbation argument due to Ko{\l}odziej  \cite{Kol96APM}.  
 By uniqueness, \cite[Theorem 0.1]{Ng16AIM}, if $\|f-g\|_p=0$ then $u=v$ and \eqref{eq: stability} holds for any $C$. Hence we can assume that $\|f-g\|_p>0$. 
 
 Let $\varphi$ be a bounded $\omega$-psh function on $X$ such that $\sup_X \varphi=0$ and 
$$
(\omega+dd^c \varphi)^n = c_f f \omega^n,
$$
where $c_f$ is a constant. The existence of $\varphi$ and $c_f$ follows from \cite[Theorem 5.8]{KN15Phong}. It follows from \cite[Proposition 2.4]{KN19} that $0<c_f$ is uniformly bounded from below.  To bound $c_f$ from above we use the Gauduchon metric as in \cite{Ng16AIM}. Let $G$ be a smooth function on $X$ such that $dd^c (e^{G} \omega^{n-1}) =0$.  It follows from the mixed Monge-Amp\`ere inequality, \cite[Lemma 1.9]{Ng16AIM} that 
$$
(\omega +dd^c \varphi) \wedge e^{G}\omega^{n-1} \geq e^{G}( c_f f )^{1/n} \omega^n. 
$$
Integrating over $X$ and using Stokes theorem we arrive at
$$
\int_X e^{G} \omega^n \geq  e^{\min_X G} \int_{X} (c_f f)^{1/n}\omega^n.
$$
Thus $c_f>0$ is uniformly bounded.
 The uniform a priori estimate in  \cite{KN15Phong} also ensures that $\varphi$ is uniformly bounded. Hence, for some uniform constant $C_1>0$ we have that 
$$
(\omega+dd^c \varphi)^n \geq e^{\varphi-C_1} f\omega^n \ ; \ (\omega+dd^c \varphi)^n \leq e^{\varphi+C_1} f \omega^n.
$$
Combining this with \cite[Lemma 2.3]{Ng16AIM}, we obtain $\varphi -C_1 \leq u\leq \varphi+C_1$, hence $u$ is also uniformly bounded by a constant $C_2$ depending on the parameters in the statement of Theorem \ref{thm: GLZ Hermitian}. By the same arguments as above, we see that $|v|\leq C_3$ for some uniform constant $C_3>0$. 
 
Let $\rho$ be the unique continuous $\omega$-psh function on $X$, normalized by $\sup_X \rho =0$, such that
\begin{equation}
\label{eq: rho}
(\omega +dd^c \rho)^n = c_{h} h \omega^n = c_h \left ( \frac{|f-g|}{\|f-g\|_p}  +  1 \right)\omega^n. 
\end{equation}
The existence of $\rho$ follows from \cite[Theorem 5.8]{KN15Phong}. 
It follows from \cite[Lemma 2.1]{KN19} that $c_h\leq 1$. Since $1\leq \|h\|_p \leq 2$, it follows from \cite[Proposition 2.4]{KN19} that $c_h \geq c_1>0$ where $c_1$ is a uniform constant.

We now set $\varepsilon:= e^{(\sup_X u-\ln c_1)/n} \|f-g\|_p^{1/n}$ and consider two cases. If  $\varepsilon>1/2$ then 
\[
\|f-g\|_p^{1/n} \geq  \frac{c_1^{1/n}}{2} e^{-\sup_X u},
\] 
hence, for $C\geq 2(C_2+C_3)c_1^{-1/n} e^{\sup_X u/n}$,  we have 
\[
|u-v| \leq C_2+C_3 \leq C \|f-g\|_p^{1/n}. 
\]
If $\varepsilon\leq 1/2$ we consider 
$$
\phi := (1-\varepsilon) u + \varepsilon \rho  -K \varepsilon + n\log(1-\varepsilon),  
$$
where $K>0$ is a constant to be specified later. The Monge-Amp\`ere measure of $\phi$ is estimated as follows: 
$$
(\omega +dd^c \phi)^n \geq e^{u + n\log (1-\varepsilon)} f \omega^n + e^{u} |f-g|\omega^n \geq e^{u+n\log(1-\varepsilon)} g \omega^n. 
$$
If we choose $K= \sup_X (-u)$ then 
$$
(\omega +dd^c \phi)^n \geq e^{\phi} g \omega^n,
$$
and Proposition \ref{prop: sub and super solutions} yields $\phi \leq v$, hence $u-v \leq C_4\varepsilon$. Reversing the role of $u$ and $v$ we obtain the result.
\end{proof}

%\subsection{Sharp stability exponent}
Using Theorem \ref{thm: GLZ Hermitian proof} we will improve the stability exponent in \cite{KN19}. 
We first prove the following refinement of \cite[Lemma 3.4]{KN19}.
\begin{lemma}
\label{lem: KN15 Lemma 3.4}
Assume that $0\leq f,g\in L^{p}\left(X\right)$ satisfy 
\begin{equation}\label{eq: stability reduced 1}
    e^{-\varepsilon} f \leq g \leq e^{\varepsilon} f,
\end{equation}
for some (small) positive constant $\varepsilon$. Let $u$ and $v$
be continuous $\omega$-psh functions on $X$ such that 
\[
\omega_{u}^{n}=f\omega^{n},\;\omega_{v}^{n}=g\omega^{n}\; \text{and}\ \sup_{X}u=\sup_{X}v=0.
\]
Fix $t_{1}>t_{0}:=\inf_{X}\left(u-v\right)$. If $\intop_{\left\{ u-v<t_{1}\right\} }f\omega^{n}\leq V_{\textrm{min}}$
then, for some uniform constant $C>0$ depending on $(X,\omega,n,p)$, an upper bound $C_p$  for $\|f\|_p$, and a positive lower bound  for $\|f\|_{1/n}$, we have  $$ t_1 - t_0 \leq C \varepsilon. $$Here
$V_{\textrm{min}}$ is the constant in Proposition \ref{prop: control mass from below} corresponding to 
$A_{0}:=2^{n}C_{p}$.
\end{lemma}

\begin{proof}
Define 
$$
\hat{f}(z) = \begin{cases} f(z), \ \text{if} \ u(z) <v(z) + t_1,\\
\frac{1}{A} f(z), \ \text{if} \ u(z) \geq v(z)+t_1, 
\end{cases}
$$
where $A>1$ is a uniform constant ensuring that $\int_X \hat{f} \omega^n < 2V_{\min}$. Let $\hat{c}>0$ be a constant and $\hat{u}$ be a continuous $\omega$-psh function such that 
$$
(\omega +dd^c \hat{u})^n = \hat{c} \hat{f} \omega^n, \ \ \sup_X \hat{u}=0.
$$
It follows from Proposition \ref{prop: control mass from below} and \cite[Corollary 2.4]{Ng16AIM} that $2^n \leq \hat{c}\leq A$, hence by \cite[Corollary 5.6]{KN15Phong}, $\hat{u}$ is uniformly bounded.

For $s\in (0,1)$, define  $\psi_s := (1-s) v + s \hat{u}$. By the mixed Monge-Amp\`ere inequality \cite[Lemma 1.9]{Ng16AIM} we have
$$
(\omega+dd^c \psi_s)^n \geq  \left ((1-s) g^{1/n} + s (\hat{c}f)^{1/n} \right )^{n} \omega^n
$$
on $\Omega(t_1):= \{u <v +t_1\}$. By the assumption \eqref{eq: stability reduced 1} and the inequality $a^{1/n} \geq a$, for $a\in (0,1)$, we have 
$$
(\omega+dd^c \psi_s)^n \geq  \left ((1-s) e^{-\varepsilon/n} f^{1/n} + s (2^n f)^{1/n} \right )^{n} \omega^n, 
$$
in $\Omega(t_1)$. 
Thus, for $s=\varepsilon$ we have $(\omega+dd^c \psi_s)^n \geq (1+\varepsilon^2/n) f \omega^n$ in $\Omega(t_1)$. As in \cite[Lemma 3.4]{KN19} we now invoke the minimum principle, Proposition \ref{prop: min principle},  to obtain 
\[
\max_{\Omega(t_1)}(\psi_s-u) =\max_{\partial \Omega(t_1)} (\psi_s-u).
\]
But on $\partial \Omega(t_1)$ we have $u=v+t_1$, hence $\psi_s - u + t_1 \leq  C_1s$ on $\partial \Omega(t_1)$, where $C_1$ is a uniform constant. Let $x_0\in X$ be such that $u(x_0)-v(x_0)=t_0$. Then $x_0 \in \Omega(t_1)$, hence  $\psi_s(x_0) -u(x_0) \leq \max_{\partial \Omega(t_1)} (\psi_s-u)$. We then infer that $t_1-t_0 \leq C s$ as desired. 
\end{proof}

\begin{prop}\label{prop: approximation}
Assume that $u$ is a continuous $\omega$-psh function such that $\omega_u^n = f\omega^n$, where $0\leq f\in L^p(X)$, $p>1$. Let $f_j>0$ be a sequence of smooth densities converging to $f$ in $L^p(X)$ and  let $u_j$ be a sequence of smooth $\omega$-psh functions decreasing to $u$. Let $v_j$ be the unique smooth $\omega$-psh function such that 
$$
\omega_{v_j}^n = e^{v_j-u_j} f_j \omega^n. 
$$
Then $v_j$ converges uniformly  to $u$. 
\end{prop}
Note that the smoothness of $v_j$ follows from \cite{Cher87}. 
\begin{proof}
Recall that, from \cite[Remark 5.7]{KN19} we have  $\int_X f^{1/n} \omega^n>0$. 
Set $F_j := e^{-u_j} f_j$ and $F:= e^{-u}f$. By \cite[Corollary 5.6]{KN15Phong}, $v_j$ is uniformly bounded. Hence $1/C \leq \int_X F_j^{1/n}$ and $\|F_j\|_p \leq C_1$, for a uniform constant $C_1$. Theorem \ref{thm: GLZ Hermitian} yields $|v_j-u| \leq C_2 \|F_j-F\|_p^{1/n}$, for a uniform constant $C_2$.  Hence $v_j$  uniformly converges  to $u$. 
\end{proof}

\begin{theorem}
Assume that $u$ and $v$ are $\omega$-psh continuous solutions to 
$$
(\omega+dd^c u)^n = f\omega^n \ ,  \ (\omega+dd^c v)^n =g\omega^n, \ \sup_X u= \sup_X v=0,
$$
where $f,g\in L^p(X), p>1$ and $f\geq c_0>0$. Then 
$$
\sup_X |u-v| \leq C \|f-g\|_p^{1/n},
$$
where $C$ depends on $X,\omega, n, p, c_0$, an upper bound for $\|f\|_p+ \|g\|_p$, and a positive lower bound for $\|g\|_{1/n}$.  
\end{theorem}

\begin{proof} For convenience we can assume that $\int_X \omega^n =1$. 
We first assume that $u,v$ are smooth and 
\begin{equation}
\label{eq: stability reduced}
e^{-\varepsilon}f \leq g \leq e^{\varepsilon}f,
\end{equation}
for some small constant $\varepsilon>0$. Then, following the proof of \cite[Theorem A]{KN19} we obtain 
$$
|u-v| \leq C \varepsilon,
$$
for a uniform constant $C>0$. The only difference compared to \cite[Lemma 3.4]{KN19} is that we can replace $\varepsilon^{\alpha}$ by $\varepsilon$ (see Lemma \ref{lem: KN15 Lemma 3.4}).  For convenience of the reader we briefly recall the arguments of \cite{KN19}. 

We set $t_0 := \min_X(u-v)$, $\hat{t}_0:= \max_X(u-v)>t_0$. Then $t_{0}\leq 0$ and  $\hat{t}_0\geq 0$. The goal is to prove that $\hat{t}_0 -t_0 \leq C \varepsilon$. Set 
$$
t_1:= \sup \left \{t\geq t_0\  ;  \ \int_{\{u<v+t\}} f \omega^n \leq V_{\min}/2 \right \},   
$$
$$
\hat{t}_1 := \inf \left \{ t \leq \hat{t}_0 \ ; \ \int_{\{u>v+t\}} f \omega^n \leq V_{\min}/2 \right \}.
$$ 
It follows from Lemma \ref{lem: KN15 Lemma 3.4} that $t_1 \leq t_0 + C\varepsilon$. Since $\varepsilon$ is small we infer that $\int_{\{v < u-t\}} g\omega^n \leq V_{\min}$, for all $\hat{t}_1< t\leq \hat{t}_0$.  It thus follows from Lemma \ref{lem: KN15 Lemma 3.4} that $-\hat{t}_1 + \hat{t}_0 \leq C\varepsilon$. Hence it remains to prove that $\hat{t}_1 -t_1 \leq C \varepsilon$. Set $s_1:=t_1+\varepsilon$ and $\hat{s}_1:=\hat{t}_1-\varepsilon$. We prove that $\hat{s}_1 -s_1 \leq C\varepsilon$.  By definition of $t_1$ and $\hat{t}_1$ we have 
$$
\int_{\{u<v+s_1\}} f \omega^n \geq V_{\min}/2 \ ; \ \int_{\{u>v+\hat{s}_1\}} f \omega^n \geq V_{\min}/2. 
$$
We choose a uniform constant $\gamma>0$ depending on $\|f\|_p$, $p$, and $V_{\min}$ such that, for all Borel set  $E\subset X$, 
$$
\int_E f \omega^n \geq V_{\min}/2 \Longrightarrow \int_{E} \omega^n \geq \gamma.
$$
The existence and uniformity of $\gamma$ follow from the H\"older inequality. 

We now use the main novelty of \cite{KN19}: estimate of the Laplacian mass on small collars (which uses the assumption $f\geq c_0>0$). 
Define $s_0 := t_0$, $s_k:= 2^{k-1}(s_1-s_0) +s_0$, for $k\geq 2$. 

If $\int_{\{u>v +s_N \}} f\omega^n \geq V_{\min}/2$, then $\int_{\{u>v+s_N\}} \omega^n \geq \gamma$ and \cite[Proposition  3.8]{KN19} applies, giving
$$
\int_{\{s_0<u-v\leq s_N\}}   \omega_u \wedge \omega^{n-1} \geq (N-1) C c_0 \gamma^{4}. 
$$
But the left hand side is uniformly bounded by a constant depending on $(X,\omega)$ (see Lemma \ref{lem: Gauduchon}). It thus follows that for $N$ large enough we have $\int_{\{u>v+s_N\}} f\omega^n < V_{\min}/2$. By definition of $\hat{t}_1$ we have $\hat{t}_1 \leq s_N$. But $s_N-s_0 \leq 2^{N-1}C \varepsilon$, hence $\hat{s}_1-s_1\leq C\varepsilon$ as desired. The first step is completed. 

\medskip

We next assume that $u,v$ are smooth but we remove the assumption \eqref{eq: stability reduced}.  
Let $w$ be the unique smooth $\omega$-psh function such that
$$
(\omega +dd^c w)^n =e^{w-v} f \omega^n =: h \omega^n.
$$
The smoothness of $w$ was proved by Cherrier \cite{Cher87}.
Since $v$ satisfies $\omega_v^n =e^{v-v} g\omega^n$, we can apply Theorem \ref{thm: GLZ Hermitian} with $F=e^{-v} f$ and $G=e^{-v}g$ and obtain 
\begin{equation}\label{eq: GLZ hermit 2}
|w-v| \leq C_1 \|f-g\|_p^{1/n},
\end{equation}
where $C_1>0$ is a uniform constant.   

We thus have  $
e^{-\varepsilon} f \leq h \leq e^{\varepsilon} f,
$ where $\varepsilon:= C_1 \|f-g\|_p^{1/n}$. The previous step yields 
$$
|w-\sup_X w - u| \leq C_2 \varepsilon. 
$$
But from \eqref{eq: GLZ hermit 2} we see that $|\sup_X w|\leq 2\varepsilon$, hence the result follows. 

We now  treat the general case. We approximate $u,v$ as in Proposition \ref{prop: approximation}. Let $u_j,v_j$ be  smooth $\omega$-psh functions decreasing to $u,v$. Let $f_j,g_j$ be smooth functions converging to $f,g$ in $L^p$ and $f_j\geq c_0/2$.    Let $\varphi_j,\psi_j$ be smooth $\omega$-psh functions solving 
$$
(\omega+dd^c \varphi_j)^n = e^{\varphi_j-u_j} f_j\omega^n\ , \ (\omega+dd^c \psi_j)^n = e^{\psi_j-v_j} g_j \omega^n. 
$$
It follows from Proposition \ref{prop: approximation} that $\varphi_j, \psi_j$ converge uniformly to $u,v$. For $j$ large enough we have $F_j:= e^{\varphi_j-u_j} f_j \geq c_0/4$. Set $G_j:= e^{\psi_j-v_j}g_j$ and observe that $\|F_j\|_p, \|G_j\|_p$ are uniformly bounded. It thus follows from the second step that 
$$
|\varphi_j-\psi_j| \leq C\|F_j-G_j\|_p^{1/n},
$$
where $C>0$ is a uniform constant. Letting $j\to +\infty$ we arrive at the result.

\end{proof}

%\subsection{Stability of the  constant}

Using the same ideas we prove a stability estimate for the MA-constant. Recall that (see \cite{KN15Phong}, \cite{KN19}) for each $0\leq f \in L^p, p>1$ with $\int_X f \omega^n >0$ there exists a unique constant $c=c_f>0$ such that the equation $\omega_u^n =c_f f \omega^n$ has a bounded weak solution in ${\rm PSH}(X,\omega)$. 

\begin{coro}
Assume that $0\leq f,g\in L^p$ for some $p>1$. Then 
\[
|c_f -c_g|\leq C \|f-g\|_p^{1/n},
\]
 where $C>0$ is a constant depending on $(X,\omega,n,p)$, an upper  bound for $\|f\|_p, \|g\|_p$, and a positive lower bound for $\|f\|_{1/n}$, $\|g\|_{1/n}$. 
\end{coro}

\begin{proof}
Let $u$ be a continuous $\omega$-psh function on $X$, normalized by $\sup_X u=0$, such that $(\omega+dd^c u)^n =c_f f\omega^n$. By Lemma \ref{lem: Gauduchon} and the mixed Monge-Amp\`ere inequality \cite[Lemma 1.9]{Ng16AIM} we have that
\[
(\omega +dd^c u) \wedge e^G \omega^{n-1} \geq c_f^{1/n} f^{1/n} e^G \omega^{n},
\]
where $G$ is a smooth function such that $dd^c(e^G \omega^{n-1}) =0$ (see \cite{Gau77}). Integrating on $X$ we see that $c_f$ is uniformly bounded from above. Proposition \ref{prop: control mass from below} then ensures that $c_f$ is uniformly bounded from below. It follows from \cite[Theorem 0.1]{Ng16AIM} that there exists a unique continuous $\omega$-psh function $v$ such that 
$$
(\omega+dd^c v)^n =e^{v-u} c_f g \omega^n. 
$$
Theorem \ref{thm: GLZ Hermitian} yields $|v-u| \leq C_1c_f \|f-g\|_p^{1/n}$, for a uniform constant $C_1$, hence
\[
\left(1-C_2 \|f-g\|_p^{1/n}\right) c_f g\omega^n  \leq (\omega +dd^c v)^n \leq \left(1+ C_2  \|f-g\|_p^{1/n}\right)  c_f g\omega^n,
\]
for some uniform constant $C_2$.  It thus follows from \cite[Lemma 2.1]{KN19} that 
\[
\left(1-C_2 \|f-g\|_p^{1/n}\right) c_f  \leq c_g \leq \left(1+ C_2  \|f-g\|_p^{1/n}\right) c_f,
\]
yielding
$$
|c_f -c_g| \leq C_2 \|f-g\|_p^{1/n},
$$ 
and concluding the proof.
\end{proof}

\subsection{The case of big cohomology classes on K\"ahler manifolds}

Using the idea of the proof of Theorem \ref{thm: GLZ Hermitian} we can also improve \cite[Theorem C]{GZ12MRL}. 
We first recall a few known facts on pluripotential theory in big cohomology classes. We refer the reader to \cite{BEGZ10,BBGZ13,DDL1,DDL2,DDL3,DDL4,DDL5} for more details. 

We assume (only in this section) that $\omega$ is K\"ahler (i.e. $d\omega =0$). Fix a closed smooth real $(1,1)$-form $\theta$.  A function $u : X \rightarrow \mathbb{R} \cup \{-\infty\}$ is $\theta$-psh if it is quasi-psh and $\theta +dd^c u\geq 0$ in the sense of currents. We let $\psh(X,\theta)$ denote the set of all $\theta$-psh functions which are not identically $-\infty$. By elementary properties of psh functions one has $\psh(X,\theta)\subset L^{1}(X)$. Here, if nothing is stated, $L^1(X)$ is $L^1(X,\omega^n)$.  The De Rham cohomology class $\{\theta\}$ is big if $\psh(X,\theta-\varepsilon \omega)$ is non-empty for some $\varepsilon>0$. 

We let $V_{\theta}$ denote the envelope: 
\[
V_{\theta} : = \sup \{u \in \psh(X,\theta) \setdef u \leq 0\}.
\]
There is a Zariski open set $\Omega$, called the ample locus of $\{\theta\}$, on which $V_{\theta}$ is locally bounded. A $\theta$-psh function $u$ has minimal singularities if $u-V_{\theta}$ is globally bounded on $X$. For a $\theta$-psh function $u$ with minimal singularities the operator $(\theta +dd^c u)^n$ is well-defined as a positive Borel measure on $\Omega$. One extends this measure trivially over $X$. The total mass of the resulting measure depends only on the cohomology class of $\theta$ and is called the volume of $\theta$, denoted by $\vol(\theta)$ (see \cite{Bou04}, \cite{BEGZ10}).  Given a $\theta$-psh function $u$,  the non-pluripolar Monge-Amp\`ere measure of $u$ is defined by 
\[
(\theta +dd^c u)^n := \lim_{j\to +\infty} {\bf 1}_{\{u>V_{\theta}-j\}} (\theta +dd^c \max(u,V_{\theta}-j))^n,
\]
where the sequence of measures on the right-hand side is increasing in $j$.  Note that $\int_X (\theta +dd^c u)^n \leq \vol(\theta)$ and the equality holds if and only if $u\in \mathcal{E}(X,\theta)$. 

It was proved in \cite{BEGZ10} that for all $L^p$-density ($p>1$), $0\leq f$ with $\int_X f \omega^n = \vol(\theta)$, there exists a unique $\theta$-psh function with minimal singularities $u$ such that $\sup_X u=0$ and $\theta_u^n = f\omega^n$.  

We assume throughout this section the following normalization: 
\begin{equation}
	\label{eq: normalization}
	\int_X \omega^n = \vol(\theta) =1. 
\end{equation}

\begin{theorem} \label{thm: improvement of GZ12}
	Under the above setting, assume that $0\leq f,g \in L^p(X,\omega^n)$, $p>1$.  If $u,v$ are $\theta$-psh functions with minimal singularities on $X$ such that 
	$$
	\theta_u^n = f\omega^n, \ \theta_v^n = g\omega^n, \ \sup_X u= \sup_X v=0,
	$$
	then for some constant $C>0$ depending on $(X,\omega,n, \theta, p)$ and an upper bound for $\|f\|_p, \|g\|_p$ we have 
	$$
	\sup_X |u-v|\leq C \|f-g\|_p^{1/n}.
	$$
\end{theorem}

In the proof below we let $C_1,C_2,...$ denote various uniform constants. 
\begin{proof}
By \cite[Theorem 4.1]{BEGZ10},  for some uniform constant $C_1>0$, we have 
\[
u \geq V_{\theta} -C_1 \; \text{and}\; v \geq V_{\theta} -C_1. 
\]
From this we get $v-C_1\leq u\leq v+ C_1$.  

By the uniform version of Skoda's integrability theorem, see \cite[Theorem 8.11]{GZbook}, \cite{Zer01}, and the H\"older inequality, there exists a small positive constant $a>0$ such that $\int_X e^{-a\varphi} f\omega^n<+\infty$ for all $\varphi\in \psh(X,\theta)$.  By \cite[Theorem 5.3]{DDL4} there exists a unique $w\in \mathcal{E}(X,\theta)$ such that 
\begin{equation}
	\label{eq: MA big}
	(\theta +dd^c w)^n = e^{a(w-v)} f\omega^n=:h\omega^{n}.
\end{equation}
Observe that $u-C_1$ (respectively $u-C_1$) is a subsolution (respectively supersolution) to the above equation, hence by  the comparison principle (see e.g. \cite[Lemma 4.24]{DDL2})  we have $u-C_1\leq w\leq u+C_1$. 

We claim that $|w-v|\leq A_1\|f-g\|_p^{1/n}$, for some uniform constant $A_1>0$. 

We set $\varepsilon := e^{2aC_1/n}\|f-g\|_p^{1/n}$ and consider two cases. If $\varepsilon>1/2$ then, by choosing $A_1= 4C_1 e^{2aC_1/n}$ we have 
\[
|w-v| \leq 2C_1\leq A_1 \|f-g\|_p^{1/n}. 
\]

Assume that $\varepsilon\leq 1/2$.   By the H\"older inequality and the normalization \eqref{eq: normalization} we can find a constant $b\geq 0$ such that 
\[
\int_X \left ( \frac{|f-g|}{\|f-g\|_p}  +b\right) \omega^n  =  \vol(\theta). 
\]
 Let $\rho\in \psh(X,\theta)$ be the unique solution with minimal singularities to 
\[
(\theta +dd^c \rho)^n = \left(\frac{|f-g|}{\|f-g\|_p} + b\right) \omega^n, \ \sup_X \rho =0.
\]

It follows from \cite[Theorem 4.1]{BEGZ10} that $\rho \geq V_{\theta}-C_3$, hence $|\rho-w|\leq C_4$. We now show that for a suitable choice of $B>0$, the function $\varphi := (1-\varepsilon) w + \varepsilon \rho -B\varepsilon$ is a subsolution to \eqref{eq: MA big}. Indeed, 
\[
\theta_{\varphi}^n \geq (1-\varepsilon)^n e^{a(w-v)} f \omega^n + e^{2aC_1}|f-g| \omega^n \geq e^{a(w-v) + n \log (1-\varepsilon) } g\omega^n.  
\]
For $B$ large enough (depending on $C_4,a$) we have that $a\varphi \leq a w + n\log (1-\varepsilon)$, hence $\varphi$ is a subsolution to $(\theta+dd^c \phi)^n =e^{a(\phi-v)} g\omega^n$.  
We thus have, by the comparison principle, that $\varphi\leq v$. Exchanging the role of $v$ and $w$ we finish the proof of the claim. 
 
We next prove that $|w-u|\leq A_2\|f-g\|_p^{1/n}$, for some uniform constant $A_2>0$. 
Since $|w-v|\leq A_1\|f-g\|_p^{1/n}$ and $\sup_X v=0$ it follows that $|\sup_X w | \leq A_1\|f-g\|_p^{1/n}$. It thus suffices to prove that ${\rm osc}_X (w-u) \leq A_2\|f-g\|_p^{1/n}$.   Replacing $w$ by $w+c$, for some constant $c$, we can assume that 
$$
\sup_X (w-u) = \sup_X (u-w) =:s\geq 0. 
$$
It is then enough to prove that $s\leq A_2 \|f-g\|_p^{1/n}$.  To do this we can assume that 
$$
2\int_{\{w<u\}} h\omega^n \leq \int_X h \omega^n
$$
(otherwise we change the role of $w$ and $u$).  Note that 
$$
\theta_w^n = h \omega^n \ ; \ \theta_u^n = f \omega^n, \ 
2^{-\delta} h \leq f \leq 2^{\delta} h,
$$
where $\delta = C_5 \|f-g\|_p^{1/n}$. Now, it suffices to prove that $u\leq w + A_2 \|f-g\|_p^{1/n}$. 
Set $U:= \{w<u\}$. 
Let $\rho$ be the unique $\theta$-psh function with minimal singularities such that 
$$
\theta_{\rho}^n = 2   h {\bf 1}_U \omega^n + b_1 \omega^n,\ \sup_X \rho =0,
$$
where $b_1\geq  0$ is a normalization constant. It follows from \cite[Theorem 4.1]{BEGZ10} that  $|\rho -u| \leq C_6$, hence
$$
V := \{w < (1-\delta) u + \delta (\rho -C_6) \} \subset U. 
$$
On $U$, the Monge-Amp\`ere measure  $\theta_{(1-\delta)u + \delta \rho}^n$  can be estimated as follows, using the mixed Monge-Amp\`ere inequalities (see \cite{BEGZ10}, \cite{Diw09Z}), 
\begin{flalign*}
\theta_{(1-\delta) u + \delta \rho}^n &\geq  \left ( (1-\delta)  f^{1/n} + \delta   (2h)^{1/n}\right )^n \omega^n\\
& \geq \left((1-\delta) 2^{-\delta/n} + 2^{1/n} \delta\right)^n h \omega^n. 
\end{flalign*}
Using the inequality $2^x =e^{x\log 2} \geq 1+ x\log 2$ we have, for $\delta \in (0,1)$, 
\begin{align*}
(1-\delta)2^{-\delta/n} + 2^{1/n} \delta &\geq (1-\delta)\left(1-\frac{\delta\log 2}{n} \right)+ \left ( 1+ \frac{\log 2}{n}\right) \delta 	\\
& =1+ \frac{\delta^2 \log 2}{n}=: 1+\gamma. 
\end{align*}
 We thus have
$$
\theta_{(1-\delta) u + \delta \rho}^n \geq (1+\gamma)h \omega^n.
$$
 The comparison principle, see \cite[Corollary 2.3]{BEGZ10}, gives 
$$
\int_{V} (1+\gamma) h \omega^n \leq  \int_{V} \theta_{(1-\delta) u + \delta \rho}^n \leq \int_V \theta_{w}^n = \int_V h \omega^n, 
$$
hence $\int_V h \omega^n =0$. Using the domination principle, see \cite[Corollary 3.10]{DDL2}, we then infer  $w \geq (1-\delta) u + \delta (\rho -C_6)$, hence $w-u \geq -C_6\|f-g\|_p^{1/n}$ which completes the proof.    
\end{proof}

\begin{rema}
In the Hermitian setting, if the comparison principle holds (which implies certain geometric conditions on $X$, see \cite{Chio16}), then the above proof can be applied. 
\end{rema}

\section{H\"older continuity}\label{sect: Holder continuity}

\subsection{H\"older regularity of solutions}

\begin{theorem}\label{thm: Holder exponent proof}
Let $(X,\omega)$ be a compact Hermitian manifold of dimension $n$. Fix  $0\leq f \in L^p(X), p>1$ with $\int_X f \omega^n >0$. Then any solution $u$ to $\omega_u^n = f \omega^n$ is H\"older continuous with H\"older exponent in $(0, p_n)$, where $p_n= 2/ (nq + 1)$.
\end{theorem}
We note here that bounded solutions to \eqref{eq: MA} are automatically continuous. Indeed, let $u$ be a bounded solution and $v$ be a continuous $\omega$-psh function such that $\omega_v^n =e^{v-u} f \omega^n$. The existence of $v$ follows from \cite{Ng16AIM}. By uniqueness $v=u$, hence $u$ is continuous. 
\begin{proof}
Assume that $u$ is a bounded $\omega$-psh function solving 
$$
(\omega +dd^c u)^n = f\omega^n.
$$
By adding a constant to $u$ we can assume that $\inf_X u=1$ and set $b:= 2\sup_X u$. 
We will use the same notations as in \cite[Section 4]{KN19}. Fix $\alpha \in (0,p_n)$.  We prove that $u$ is H\"older continuous with exponent $\alpha$ by showing that $\rho_t u -u \leq c t^{\alpha}$, for $t$ small enough (see \cite[page 632]{DDGKPZ14}, \cite[Lemma 4.2]{GKZ08} or Lemma \ref{lem: GKZ08} below). Here, following Demailly \cite{Dem92}, $\rho_{t}(u)$ is defined by
\begin{equation}\label{eq: Demailly}
\rho_t(u)(z) := \frac{1}{t^{2n}} \int_{T_zX}  u({\rm exph}_z (\zeta) ) \rho \left ( \frac{\|\zeta\|_{\omega}^2}{t^2} \right ) dV_{\omega}(\zeta),
\end{equation}
where $\zeta \mapsto {\rm exph}_z(\zeta)$ is the (formal) holomorphic part of the Taylor expansion of the exponential map of the Chern connection on the tangent bundle of $X$ associated to $\omega$, and $\rho$ is a smoothing kernel defined by 
$$
\rho(t) :=    \begin{cases} \frac{\eta}{(1-t)^2} {\rm exp} \left ( \frac{1}{t-1} \right ) , \ \text{if}\ t \in [0,1], \\
0 , \ \text{if} \ t>1, \end{cases}
$$
where $\eta>0$ is a constant such that $\int_{\mathbb{C}^n}\rho(\|z\|^2) dV(z) =1$. Here $dV$ is the Lebesgue measure on $\mathbb{C}^n$. 

Following  \cite{Dem92} and \cite{KN19},  we define the  Kiselman-Legendre transform: 
\begin{equation}\label{eq: KL}
U_{\delta,c}:= \inf_{t\in [0,\delta]} \left (\rho_t(u) + K(t^2-\delta^2) +K(t-\delta) - c \log (t/\delta) \right), 
\end{equation}
where $c>0, \delta>0$, and $K$ is a positive (curvature) constant and as in \cite{Dem92} we choose $K$ to ensure that $t \mapsto \rho_t(u) + K t^2$ is increasing in $t$. 
In the following arguments we choose $c= \delta^{\alpha}$ and we write $U_{\delta}$ instead of $U_{\delta,c}$.  It follows from \cite[Lemma 4.1]{KN19} that 
$$
\omega + dd^c U_{\delta} \geq -A \delta^{\alpha} \omega,
$$
where $A>0$ is a uniform curvature constant. Setting
$$
u_{\delta} := \frac{1}{1+2A\delta^{\alpha}} U_{\delta}, 
$$
we then have $\omega +dd^c u_{\delta} \geq \gamma \omega$, for some positive constant $\gamma$. Note that by construction and by the choice of $K$, we have 
$$
\rho_{\delta}(u) + K\delta^2 \geq u, \  \text{and}\ \rho_{\delta}(u) \geq U_{\delta}. 
$$
Set $s := e^{-5Ab}$ and 
$$
E(\delta) := \{\rho_{\delta}(u) -u > Ab \delta^{\alpha}\}, \ F(\delta):= \{\rho_{s\delta} u \geq u + 5 Ab \delta^{\alpha}\}.
$$
Up to decreasing $\delta$ we can assume that $2K \delta \leq Ab \delta^{\alpha}$.  We claim that on $F(\delta)$ we have $U_{\delta} -u \geq 4 Ab \delta^{\alpha}$.  Indeed, since $t\mapsto \rho_t u + Kt^2$ is increasing and $s$ is small, we have 
$$
\rho_t(u) + K(t^2-\delta^2) +K(t-\delta) - \delta^{\alpha} \log (t/\delta) \geq  u - 2K  \delta + 5Ab\delta^{\alpha},  \ \forall t \in [0,s\delta],
$$
and 
$$
\rho_t(u) + K(t^2-\delta^2) +K(t-\delta) - \delta^{\alpha} \log (t/\delta) \geq  \rho_{s\delta} u - 2K  \delta,  \ \forall t \in [s\delta,\delta].
$$
It thus follows that on $F(\delta)$ we have $U_{\delta} \geq u + 4Ab \delta^{\alpha}$, as claimed.

Now we prove that the set $F(\delta)$ is empty for $\delta>0$ small enough. It follows from \cite[eq (4.9)]{KN19} (which is a lemma in \cite{DDGKPZ14}) that  
$$
\int_X (\rho_t u - u) \omega^n \leq C t^2. 
$$
Hence 
$$
\int_{E(\delta)} \omega^n \leq \dfrac{C}{Ab} \delta^{2-\alpha},
$$
and an application of the H\"older inequality yields
$$
\int_{E(\delta)} f\omega^n \leq   C_1 \delta^{\beta},
$$
where  $\beta:=(2-\alpha)/q$, and $q$ is the conjugate of $p$. 

 We let $v$ be the unique continuous $\omega$-psh function such that 
$$
(\omega +dd^c v)^n = e^{v-u} f {\bf 1}_{X\setminus E(\delta)} \omega^n.
$$
Theorem \ref{thm: GLZ Hermitian} yields, for each $\varepsilon>0$, $|v-u| \leq C_3 \delta^{\beta/(n+\varepsilon)}$, where $C_3$ depends also on $\varepsilon$.  Since $\alpha<p_n$ we can choose $\varepsilon >0$ so small that $\beta/(n+\varepsilon) > \alpha$.  Decreasing $\delta$ we can ensure that $|v-u|\leq Ab \delta^{\alpha}/2$. The choice of $b$ ensures that 
$$
|u_{\delta}-U_{\delta}| \leq \frac{Ab\delta^{\alpha}}{2}. 
$$
Assume by contradiction that $F(\delta)\neq \emptyset$. On $F(\delta)$ we have 
$$
u_{\delta} - v= u_{\delta}-U_{\delta} + U_{\delta}-u  + u-v \geq 3Ab\delta^{\alpha},
$$
while on $X\setminus E(\delta)$, we have 
$$
u_{\delta} -v = u_{\delta} - U_{\delta} + U_{\delta} - \rho_{\delta}u +\rho_{\delta} u - u \leq 2Ab \delta^{\alpha}.
$$
It thus follows that $u_{\delta}-v$ attains its maximum over $X$ at some point $z_0\in E(\delta)$, contradicting the minimum principle (Proposition \ref{prop: min principle}) since $\omega_v^n =0< \omega_{u_{\delta}}^n$ on $E(\delta)$. Hence, for $\delta$ small enough,  $F(\delta)$ is empty. This completes the proof.
\end{proof}

\subsection{H\"older regularity of plurisubharmonic envelopes}
For a continuous function  $f: X\rightarrow  \mathbb{R}$ we define its $\omega$-psh envelope by:
\[
P_{\omega}(f):= (\sup \{\phi \ | \  \phi\in \psh(X,\omega) \text{ and } \phi\leq f  \})^*.
\]
Is was proved in \cite{Tos18MRL} (for the K\"ahler case) and in \cite{CZ19China} (for the Hermitian case) that $P_{\omega}(f)$ belongs to $C^{1,1}(X)$ if $f$ is smooth.   If $f$ is (Lipschitz) continuous  then $P_\omega(f)$ is also (Lipschitz) continuous \cite{CZ19China}.  In this section we prove that $P_\omega(f)$ is H\"older continuous provided that $f$ is H\"older continuous. 
\begin{lemma}
If $f\in C^{0}(X)$ then $P_\omega(f) \in C^{0}(X) $.
\end{lemma}
\begin{proof}
Let $f_j\in C^{\infty}(X)$ be a sequence of smooth functions which converges uniformly to $f$. 
Since  $P_\omega(f_j)$ is continuous and 
\[
\| P_\omega(f_j)- P_\omega(f)\|_{L^{\infty}(X)} \leq \|f_j-f\|_{L^{\infty}(X)},
\] 
we see that $P_{\omega}(f)$ is continuous. 
\end{proof}

\begin{theorem}\label{thm: Holder envelope}
Assume that $f\in C^{0,\alpha}(X)$ for some $\alpha\in (0,1)$. Then $P_\omega(f)\in C^{0,\alpha}(X)$.
\end{theorem}

\begin{proof}
It follows from Choquet's lemma and the definition of the psh envelope that there exists a  sequence of $\omega$-psh functions $(\phi^j)_{j\in \mathbb{N}}$ such that $P_{\omega}(f)=(\sup_j{\phi^j})^*$ , $\phi^j\leq f$ and $\| \phi^j\|_\infty\leq C(\|f\|_\infty) $. Replacing $\phi^j$ by $(\sup_{k\leq j} \phi^k)^*$ we can assume that $\phi^j \nearrow P_\omega(f)$. Since $P_{\omega}(f)$ is continuous on $X$ we also have, by Dini's theorem, that $\phi^j$ converges uniformly to $P_{\omega}(f)$. 

For a $\omega$-psh function $u$ we consider the convolution $\rho_{t}u$ defined as in \eqref{eq: Demailly} and the Kiselman-Legendre transform defined as in \eqref{eq: KL}. Since $\phi^j\leq f$ and $f\in C^{0,\alpha}(X)$, we have 
\begin{equation}\label{ineq_1}
\rho_{\delta} \phi^j \leq \rho_{\delta} f\leq f+ C\delta^\alpha\|f\|_{0,\alpha},
\end{equation}
where $C$ depends only on $X,\omega$.

We now use  the  Kiselman-Legendre transform $\Phi^j_{\delta,c}:= \Phi_{\delta,c}(\phi^j)$.  
From  \eqref{eq: KL}, with $t=\delta$, we have that $ \Phi^j_{\delta,c} \leq \rho_\delta  \phi^j$. It follows from \cite[Lemma 4.1]{KN19} that 
\[
\omega+dd^c \Phi^j_{\delta, c}\geq- (Ac+2K\delta)\omega,
\]
where $A$ is a positive curvature constant.

We now fix  $c=(\delta^\alpha -2K\delta)/A$ so that $Ac+2K\delta =\delta^\alpha$.  We have 
$$
\omega + dd^c \Phi^j_{\delta,c} \geq - \delta^{\alpha} \omega.
$$
Setting
$$
\varphi^j_{\delta} := (1-\delta^\alpha)\Phi^j_{\delta,c}
$$
we then have $\omega +dd^c\varphi^j_{\delta}  \geq \delta^{2\alpha} \omega$ and $ \|\varphi^j_{\delta} - \Phi^j_{\delta,c}\|\leq C_0 \delta^\alpha,$ where $C_0$ depends on $\|\phi^j\|_\infty$.
From \eqref{ineq_1} and the fact that $\Phi^j_{\delta,c} \leq \rho_{\delta} \phi^j$, we infer 
$\varphi^j_{\delta} -C_1\delta^\alpha  \leq  f$, where $C_1$ depends only on $|f\|_{0,\alpha}$, $\|\phi^j\|_\infty$ and $A$.  Therefore we get
\begin{equation}
\varphi^j_{\delta}-C_1\delta^{\alpha}   \leq P_{\omega}(f)
\end{equation}
 by the definition of $ P_\omega(f)$ and the fact that  $\varphi^j_{\delta}$ is  $\omega$-psh.
 This implies that 
 \begin{equation}\label{ineq_Phi}
  \Phi^j_{\delta,c}- P_\omega(f)\leq  C_2\delta^\alpha,
 \end{equation}
 where $C_2$  depending only $C_1$ and $\|\phi^j\|_\infty$. Since $\phi^j$ converges uniformly to $P_{\omega}(f)$ we infer
 \begin{equation}\label{ineq_key}
 \Phi^j_{\delta,c}-\phi^j\leq  2C_2\delta^\alpha,
 \end{equation}
 for $j$  sufficiently large. 

Following \cite{KN18AIF} we now use \eqref{ineq_key} to estimate  $\rho_\delta \phi^j- P_\omega(f)$.
For any $x\in X$,  the minimum in the definition of $\Phi_{\delta,c}^j$ achieves at $t_0=t_0(x,j)$. It follows from \eqref{ineq_key} that  
\begin{eqnarray}
 \rho_{t_0}\phi^j +K(t_0-\delta)+K(t_0^2-\delta^2) -c\log (t_0/\delta) -\phi^j \leq C_3\delta^\alpha, 
\end{eqnarray}
where $C_3$ depends only on $\|f\|_{0,\alpha}$, $\|\phi^j\|_\infty$.
 Since $\rho_t \phi^j +Kt^2+Kt -\phi^j\geq 0$, we have 
 $$c\log \frac{t_0}{\delta}\geq -C_4\delta^\alpha.$$
For $\delta$ small enough we have  $c\geq \delta^\alpha/(2A)$, hence 
\begin{equation}
t_0\geq a \delta, \quad \text{for}\,\,  a=e^{-2AC_4}.
\end{equation}
Since $\rho_t+ Kt^2+Kt$ is increasing in $t$ and $t_0\geq a \delta$, we infer
\begin{eqnarray*}
\rho_{a \delta}\phi^j +Ka \delta+K(a \delta)^2 -P_\omega(f)&\leq&\rho_{t_0}\phi^j +Kt_0+Kt_0^2-P_\omega(f)\\
&\leq &\Phi^j_{\delta, c}- P_\omega(f) - c\log a\\
&\leq & C_5\delta^\alpha,
\end{eqnarray*}
where $C_5$ depends only on $\|f\|_{0,\alpha}$, $\|\phi^j\|_\infty$, $K$, $A$, and in the last line we have used \eqref{ineq_Phi}. 
Since $\phi^j \nearrow P_\omega
(f )$, we have that $\rho_\delta \phi^j $ converges to $\rho_\delta P_\omega(f) $ as $j\rightarrow \infty$.   Therefore, letting $j$ tend to $\infty$, and then replacing $a\delta$ by $\delta$ we get 
\begin{equation} \label{ineq:lower_bound}
\rho_\delta  P_{\omega}(f) -C_6\delta^\alpha \leq P_\omega( f),
\end{equation}
where $C_6$ depends only on $\|f\|_{0,\alpha}$, $K$ and $A$. 
Invoking Lemma \ref{lem: GKZ08} below we conclude that $P_{\omega}(f) \in {\rm Lip}_{\alpha}(X)$. 
\end{proof}

\begin{lemma}\label{lem: GKZ08}
	Assume that $u$ is a bounded $\omega$-psh function on $X$ such that $\rho_t u \leq u + C_0t^{\alpha}$ for some positive constants $C_0$ and $0<\alpha <1$. Then $u\in {\rm Lip}_{\alpha}(X)$. 
\end{lemma}
The proof of the lemma was implicitly written in \cite{DDGKPZ14}, \cite{GKZ08}. We include it for completeness. 
\begin{proof}
	We can assume that $u\leq 0$. Let $d$ be the Riemann distance on $X$ induced by  the metric $\omega$. Define 
	\[
	\tau(\delta) := \sup \{|u(x)-u(y)| \setdef x,y \in X, d(x,y) \leq \delta\}, \ \delta>0. 
	\]
	We assume by contradiction that $\limsup_{\delta \to 0^+} \delta^{-\alpha}\tau(\delta) =+\infty$. For each $\delta>0$ we can find $x_{\delta}\in X$, $y_{\delta}\in X$ such that $d(x_{\delta},y_{\delta}) \leq \delta$ and $\tau(\delta) = u(y_{\delta})- u(x_{\delta}) >0$. We can thus find  $x_0\in X$ and a sequence $\delta_j\searrow 0$ such that  
	\[
	\lim_{j\to+\infty}d(x_{\delta_j},x_0)=\lim_{j\to +\infty} d(y_{\delta_j},x_0) =0 \; \text{and} \lim_{j\to +\infty} \delta_j^{-\alpha}|u(x_{\delta_j})-u(y_{\delta_j})|=+\infty.
	\]
	 Let $B\subset X$ be a small ball around $x_0$ which will be identified with  the unit ball $\mathbb{B}$ of $\mathbb{C}^n$ via a biholomorphism. Up to adding a smooth function we can now view $u$ as a psh function in $\mathbb{B}$ and $d(x,y) \simeq  \|x-y\|$ for $x,y\in \mathbb{B}$. It follows from \cite[Remark 4.6]{Dem94} that $\rho_r u(x_{\delta}) = u\star  \rho_r (x_{\delta}) + O(r^2)$. 
	Fix $b>1$ so large that 
	\[
	(b+2)^{\alpha} \left (1- \frac{b^{2n}}{(b+1)^{2n}} \right ) < \frac{1}{2}. 
	\]
	Fix $\delta>0$ so small that $2(b+1) \delta <1$. For $\xi\in \mathbb{B}$ we denote (see \cite{Dembook}, page 32) 
	\[
	\mu_S(u;\xi,r) := \frac{1}{\sigma_{2n-1}} \int_{\mathbb{S}} u(\xi +rx) d\sigma(x), 
	\]
	\[
	\mu_B(u;\xi,r) := \frac{1}{V_{2n}r^{2n}} \int_{B(\xi,r)} u(x) dV(x). 
	\]
	Here, $\sigma$ is the area measure of the unit sphere $\mathbb{S}=\partial \mathbb{B}$, $\sigma_{2n-1}= \sigma(S(0,1))$, $V_{2n} = \vol(\mathbb{B})$.  
	Note that $\mu_S\geq \mu_B$ and these are non-decreasing in $r$. 
	By the mean value inequality we have that, for  $r=(b+1)\delta$, 
	\begin{flalign*}
		\mu_B (u;x_{\delta},r) &=  \frac{1}{V_{2n} r^{2n}} \int_{B(x_{\delta},r)} u(x) dV(x)\\
		&= \frac{1}{V_{2n} r^{2n}}  \left ( \int_{B(y_{\delta}, b\delta)} u(x) dV(x) +   \int_{B(x_{\delta},r) \setminus B(y_{\delta}, b\delta)} u(x) dV(x) \right )\\
		&\geq  \frac{b^{2n}}{(b+1)^{2n}} u(y_{\delta})  + \left (1-  \frac{b^{2n}}{(b+1)^{2n}} \right ) (u(y_{\delta}) -\tau(r+\delta))\\
		&= u(y_{\delta}) -  \left (1-  \frac{b^{2n}}{(b+1)^{2n}} \right ) \tau((b+2)\delta). 
	\end{flalign*}
	Since $\mu_S(u;x_{\delta},t) -u(x_{\delta}) \geq 0$ and non decreasing in $t>0$, we have 
	\begin{eqnarray*}
	 u\star \rho_{2r} (x_{\delta}) -u(x_{\delta}) & = &\sigma_{2n-1} \int_0^1 (\mu_S (u;x_{\delta},2tr)-u(x_{\delta})) t^{2n-1} \rho(t) dt\\
		&\geq & (\mu_S(u;x_{\delta},r) -u(x_{\delta})) \sigma_{2n-1} \int_{\frac{1}{2}}^1 t^{2n-1}\rho(t)dt\\
		&\geq & C_2 (\mu_B(u;x_{\delta},r) -u(x_{\delta}))\\
		&\geq & C_2\left (\tau(\delta)  - \frac{1}{2} \frac{\tau((b+2)\delta)}{(b+2)^{\alpha}} \right). 
	\end{eqnarray*}
	Using these estimates and the assumption that $\rho_{2r}u(x_{\delta}) \leq u(x_{\delta}) + C_0(2r)^{\alpha}$ we arrive at 
	\[
	\tau(\delta) - \frac{1}{2} \frac{\tau((b+2)\delta)}{(b+2)^{\alpha}} \leq  C_5 \delta^{\alpha}. 
	\]
	We set $h(\delta)= \delta^{-\alpha} \tau (\delta)$. For $\delta>0$ small enough, say $\delta \in (0,\varepsilon_0]$ for some $\varepsilon_0>0$ fixed, and $c=b+2$ we then have  
	\[
	h(\delta)  \leq \frac{1}{2} h(c\delta)+ C_5.
	\]
	Applying this several times we obtain, for all $k \in \mathbb{N}$ with $c^{k-1}\delta\leq \varepsilon_0$,
	\begin{equation}\label{eq: GKZ08 13052020}
	h(\delta) \leq 2^{-k} h(c^k\delta) + 2 C_5, \; k \in \mathbb{N}. 
	\end{equation}
	We are now ready to derive a contradiction. We set 
	$$
	C_6:= \sup_{\delta\in [\varepsilon_0,c\varepsilon_0]} h(\delta)<+\infty.
	$$  
We have assumed that there exists a sequence $\delta_j\searrow 0$ such that $h(\delta_j)\to +\infty$. Take $j$ so large that $\delta_j<\varepsilon_0$ and $h(\delta_j) >2C_5+C_6+1$.  We choose $k\in \mathbb{N}$ such that 
\[
\frac{\log(\varepsilon_0/\delta_j)}{\log c}  \leq k \leq  \frac{\log(\varepsilon_0/\delta_j)}{\log c} + 1. 
\] 
Then  $c^k \delta_j \in [\varepsilon_0,c\varepsilon_0]$.  
From this and \eqref{eq: GKZ08 13052020}  we obtain $h(\delta_j) \leq 2^{-k}h(c^k\delta_j) +2C_5\leq 2C_5+C_6$, a contradiction. 	
\end{proof}

\bibliographystyle{//Users/lu/Desktop/Bib/amsplain_nodash.bst}
\bibliography{//Users/lu/Desktop/Bib/Biblio.bib}

\end{document}